\documentclass[12pt]{amsart}
\usepackage[margin=1.2in]{geometry}
\usepackage{graphicx,latexsym}
\usepackage{tikz}
\usepackage{amsfonts, amssymb, amsmath, amsthm, bm, hyperref}

\newcommand{\N}{\mathbb{N}}
\newcommand{\Z}{\mathbb{Z}}
\newcommand{\R}{\mathbb{R}}
\newcommand{\C}{\mathbb{C}}

\usetikzlibrary{arrows,chains,matrix,positioning,scopes}
\makeatletter
\tikzset{join/.code=\tikzset{after node path={%
\ifx\tikzchainprevious\pgfutil@empty\else(\tikzchainprevious)%
edge[every join]#1(\tikzchaincurrent)\fi}}}
\makeatother
\tikzset{>=stealth',every on chain/.append style={join},
         every join/.style={->}}
\tikzstyle{labeled}=[execute at begin node=$\scriptstyle,
   execute at end node=$]

\newtheorem{theorem}{Theorem}[section]
\newtheorem{proposition}[theorem]{Proposition}

\newtheorem{lemma}[theorem]{Lemma}
\newtheorem{corollary}[theorem]{Corollary}

\theoremstyle{definition}

\theoremstyle{remark}
\newtheorem{remark}[theorem]{Remark}

\numberwithin{equation}{section}

\begin{document}
\title[Vector-valued Eidelheit sequences]{Vector-valued Eidelheit sequences and the non B-completeness of tensor products}

\author[A. Debrouwere]{Andreas Debrouwere}
\address{Department of Mathematics, Ghent University, Krijgslaan 281, 9000 Gent, Belgium}
\thanks{A. Debrouwere gratefully acknowledges support by Ghent University, through a BOF Ph.D.-grant.}
\email{Andreas.Debrouwere@UGent.be}

\author[J. Vindas]{Jasson Vindas}
\thanks{J. Vindas gratefully acknowledges support by Ghent University, through the BOF-grant 01N01014.}
\address{Department of Mathematics, Ghent University, Krijgslaan 281, 9000 Gent, Belgium}
\email{Jasson.Vindas@UGent.be}

\begin{abstract}
We introduce vector-valued Eidelheit sequences and obtain a characterization that generalizes Eidelheit's classical theorem \cite{Eidelheit}. As an application, we discuss criteria for non $B$-completeness of completed tensor products.
\end{abstract}
\maketitle
\section{Introduction}
In 1936 Eidelheit \cite{Eidelheit} (see also \cite[Thm.\ 26.27]{M-V}) characterized the sequences $(x'_n)_{n \in \N}$ of continuous linear functionals on a Fr\'echet space $E$ such that for each arbitrary sequence $(a_n)_{n \in \N}$ of complex numbers the infinite system of linear equations
\begin{equation}
\langle x'_n,x\rangle = a_n, \qquad n \in \N,
\label{scalar-problem}
\end{equation}
has solutions $x \in E$. In his memory such sequences are nowadays called \emph{Eidelheit sequences}. Later on, this notion has been extensively studied by Mityagin \cite{Mityagin} and Vogt \cite{Vogt-86, Vogt-89, Vogt-95}.

In this article we shall study the corresponding \emph{vector-valued} Eidelheit problem. Our results reveal an interesting connection between Eidelheit sequences and Pt\'ak's notion of $B$-completeness. Let $E$ be a Hausdorff locally convex space. For each $n\in\mathbb{N}$ consider a continuous linear mapping $T_n$ from $E$ into a Hausdorff locally convex space $F_n$. We write $T_n(x)=\langle T_n, x\rangle$. Suppose that $E$ is $B$-complete and each $F_n$ is barreled. In our main theorem we shall give necessary and sufficient conditions over the linear mappings $(T_n)_{n \in \N}$  such that for any arbitrary choice of $y_n\in F_n $, $n\in\mathbb{N}$, the infinite system of linear equations
\begin{equation}
\langle T_n,x \rangle  = y_n, \qquad n \in \N,
\label{abstractmomentproblem}
\end{equation}
admits solutions $x \in E$. In the scalar-valued case $(F = \C)$ our conditions are equivalent to those in Eidelheit's theorem. Hence, we extend here his characterization to $B$-complete spaces.

As a rather surprising application, we shall also provide criteria on a Fr\'echet space and a reflexive $(DF)$-space which ensure that their completed tensor product (with respect to the $\varepsilon$-, projective, biequicontinuous, or inductive topology) is \emph{not} $B$-complete. The latter is partly based on the work of Valdivia \cite{Valdivia-87}. As a corollary, we show that the space $\mathcal{O}_C(\R^d)$ of very slowly increasing functions \cite{Horvath} is not $B$-complete.
\section{Preliminaries}
Throughout this article every locally convex space (from now on abbreviated as l.c.s.) is assumed to be Hausdorff. 
 Whenever we want to specify a topology on its topological dual $E'$ we will employ the subscript $\sigma$ for the weak$^*$ topology, $\tau$ for the Mackey dual topology, or $b$ for the strong topology. The polar in $E'$ of a set $U \subseteq E$ is denoted by $U^\circ$.  
For a l.c.s.\ $F$, we write $L(E, F)$ for  the space of all continuous linear mappings from $E$ into $F$. As in the introduction, we often write $T(x)=\langle T, x\rangle$, where $T\in L(E,F)$. Cartesian products and direct sums of l.c.s.\ are topologized in the standard fashion \cite{Kothe}.

One of the central notions in this article is the one of $B$-completeness \cite[Chap.\ 3, Sect.\ 17]{Horvath}, \cite[Chap.\ 7]{Bonet}. A l.c.s.\ $E$ is said to be $B$-complete (or a Pt\'ak space) if a linear subspace $X \subseteq E'$ is closed in $E'_\sigma$ if and only if $X \cap A$ is closed in $A$  (endowed with the relative topology of $E'_\sigma$) for every equicontinuous set $A \subset E'$. Every Fr\'echet space is $B$-complete by the Krein-$\check{S}$mulian theorem \cite[Chap.\ 3, Thm.\ 10.2]{Horvath}. Moreover, every reflexive $(DF)$-space is $B$-complete \cite[Chap.\ 3, Prop.\ 17.6]{Horvath}. 

Concerning topological tensor products we shall use the following terminology \cite{Bonet, Grothendieck, Treves}. Let $E$ and $F$ be locally convex spaces. The tensor product $E \otimes F$ is canonically isomorphic to the space $B(E'_\sigma, F'_\sigma)$ of jointly continuous bilinear functionals on $E'_\sigma \times F'_\sigma$ via 
$$
(x \otimes y)(x', y') = \langle x', x\rangle \langle y', y\rangle, \qquad x' \in E',\ y' \in F'.
$$
The $\varepsilon$-topology on $E \otimes F$ is defined as the topology carried over from $B(E'_\sigma, F'_\sigma)$ endowed with the topology of uniform convergence on the products of equicontinuous subsets of $E'$ and $F'$. The $\pi$- (resp., $\beta$-, $i$-) topology on $E \otimes F$ is the finest locally convex topology on this vector space for which the canonical bilinear mapping $E \times F \rightarrow E \otimes F$ is jointly (resp., hypo-, separately) continuous. 
 Let $t = \varepsilon, \pi, \beta$ or $i$. Endowed with the $t$-topology, the space $E \otimes F$ will be denoted by $E \otimes_t F$ and its completion by  $E \widehat{\otimes}_{t} F$. Let $G$ and $H$ be l.c.s.\ and let $T: E \rightarrow G$, $S: F \rightarrow H$ be continuous linear mappings. The associated tensor mapping $T \otimes S: E \otimes_t F \rightarrow G \otimes_t H$ is continuous. We write $S \widehat{\otimes}_t  T$ for the extension of $S \otimes T$ as a continuous linear mapping from $E \widehat{\otimes}_t  F$ into $G \widehat{\otimes}_t  H$. 

\section{Vector-valued Eidelheit sequences}
Let $E$  be a locally convex space. For each $n\in\mathbb{N}$, we consider $T_n\in L(E, F_n)$, where $F_n$ is a locally convex space. We seek conditions over $(T_n)_n$ such that for each $(y_n)_n\in \prod_{n\in\mathbb{N}} F_n$ the infinite system of equations (\ref{abstractmomentproblem}) admits solutions $x \in E$. If this is the case, we call $(T_n)_n$ an \emph{Eidelheit sequence} (or, in short, \emph{Eidelheit}). To the system of equations \eqref{abstractmomentproblem} we associate the mapping
$$
\Lambda: E \rightarrow \prod_{n\in\mathbb{N}} F_n:\ x \rightarrow (\langle T_n,x \rangle)_n.
$$
Naturally, $(T_n)_n$ is Eidelheit if and only if $\Lambda$ is surjective.

\begin{theorem}\label{vector-valued}
Let $E$ be a l.c.s., $(F_{n})_{n}$ be a sequence of l.c.s., and $(T_n)_n \in \prod_{n\in\mathbb{N}}  L(E,F_n)$. If  $(T_n)_n$ is Eidelheit, then the following three properties hold:
\begin{itemize}
\item[\textbf{(P1)}] For every $N \in \N$ and $y'_0\in F'_0, \ldots, y'_N \in F'_N$
$$
\sum_{n = 0}^N y'_n \circ T_n = 0
$$
implies $y'_0 = \ldots = y'_N = 0$.
\item[\textbf{(P2)}] For every equicontinuous set $A \subset E'$ there is $\nu \in \N$ such that for every $N > \nu$ and $y'_0\in F'_{0},\ldots, y'_N \in F'_{N}$
$$
\sum_{n = 0}^N y'_n \circ T_n \in A
$$
implies $y'_{\nu +1} = \ldots = y'_N = 0$.
\item[\textbf{(P3)}] For every $N \in \N$ the mapping
$$
E \rightarrow \prod_{n=0}^{N} F_n: x \rightarrow ( \langle T_0,x \rangle, \ldots ,  \langle T_N,x \rangle)
$$
is surjective.
\end{itemize}
Conversely, suppose that $E$ is $B$-complete and each $F_n$ is barreled. If $(T_n)_n$ satisfies \textnormal{\textbf{(P1)}}, \textnormal{\textbf{(P2)}}, and \textnormal{\textbf{(P3)}},
then $(T_n)_n$ is Eidelheit. 
\end{theorem}
\begin{proof}
The transpose of $\Lambda$ is given by 
$$
\Lambda^t: =  \bigoplus_{n\in\mathbb{N}} F'_{n} \rightarrow E':   \ (y'_0, \ldots, y'_N) \rightarrow \sum_{n = 0}^N y'_n \circ T_n = \sum_{n=0}^{N} T^{t}_{n}(y'_n).
$$

Suppose that  $(T_n)_n$ is Eidelheit. The property \textbf{(P3)} trivially holds. Since $\Lambda$ is continuous and surjective, $\Lambda^t$ is a weak$^*$-topological isomorphism onto its image. 
In particular, $\textbf{(P1)}$ holds. Next, we actually show that a stronger property than \textbf{(P3)} holds, namely, that it is satisfied not only for equicontinuous sets but also for any bounded subset of $E'_{\sigma}$. So, let $A$ be a bounded set in $E'_{\sigma}$.  The set $B = (\Lambda^t)^{-1}(A)$ is weakly$^*$ bounded, and thus bounded in the Mackey dual topology, as follows from Mackey's theorem \cite[p.\ 254]{Kothe}. We have that $(\prod_{n\in\mathbb{N}} F_{n})'_{\tau}=\bigoplus_{n\in\mathbb{N}} (F_{n})'_{\tau}$  \cite[p.\ 286]{Kothe}. Hence $B \subset 
\bigoplus_{n  = 0}^\mu F'_n$ for some $\mu \in \N$ \cite[p.\ 213]{Kothe}.

Conversely, suppose that $E$ is $B$-complete, each $F_n$ is barreled, and that $(T_n)_n$ satisfies \textbf{(P1)}, \textbf{(P2)}, and \textbf{(P3)}. From \textbf{(P1)} and the Hahn-Banach theorem we conclude that $\Lambda(E)$ is dense in $\prod_{n\in\mathbb{N}}F_n$ and therefore we only need to show that $\Lambda(E)$ is closed in $\prod_{n\in\mathbb{N}}F_n$. Since the quotient of a $B$-complete space with a closed subspace is again $B$-complete \cite[Chap.\ 3, Prop.\ 17.2]{Horvath} and $B$-complete spaces are always complete \cite[Prop.\ 17.3(b)]{Horvath}, it suffices to show that $\Lambda$ is a topological homomorphism. Hence, in view of \cite[Lemma 37.7]{Treves}, we need to show that $(i)$ $\Lambda^t(\bigoplus_{n\in\mathbb{N}} F'_{n})$ is closed in $E'_\sigma$ and that $(ii)$ for every equicontinuous set $A \subset E'$ there is an equicontinuous set $B \subset \bigoplus_{n\in\mathbb{N}} F'_{n}$ such that $\Lambda^t(\bigoplus_{n\in\mathbb{N}} F'_{n}) \cap A \subseteq \Lambda^t(B)$.  We first show $(i)$. Since $E$ is $B$-complete, we have to show that $\Lambda^t(\bigoplus_{n\in\mathbb{N}} F'_{n}) \cap A$ is closed in $A$ for every equicontinuous set $A \subset E'$. By \textbf{(P2)} there is $\nu \in \N$ such that $\Lambda^t(\bigoplus_{n\in\mathbb{N}} F'_{n}) \cap A =  \Lambda_\nu^t\left(\oplus_{n  = 0}^\nu F'_n\right) \cap A$, where 
$$
\Lambda_\nu: E \rightarrow \prod_{n=0}^\nu F_n:\  x \rightarrow ( \langle T_0,x \rangle, \dots ,   \langle T_\nu,x \rangle).
$$
The property \textbf{(P3)} says that the mapping $\Lambda_\nu$ is surjective, the Pt\'ak open mapping theorem \cite[Chap.\ 3, Prop.\ 17.2]{Horvath}  then implies that 
$\Lambda_\nu$ is a topological homomorphism. Consequently,  $\Lambda_\nu^t\left(\bigoplus_{n  = 0}^\nu F'_n\right)$  is closed in $E'_\sigma$, which yields the assertion.
Finally, let us verify $(ii)$. Fix an equicontinuous subset $A$ of $E'$. By \textbf{(P2)} there is $\nu \in \N$ such that $\Lambda^t(\bigoplus_{n \in\mathbb{N}} F'_n) \cap A =  \Lambda_\nu^t\left(\bigoplus_{n  = 0}^\nu F'_n\right) \cap A$. Since $\Lambda_\nu$ is a topological homomorphism, there is an equicontinuous subset $B$ of $\bigoplus_{n  = 0}^\nu F'_n$, and thus a fortiori of $\bigoplus_{n \in\mathbb{N}} F'_n$, such that $\Lambda_\nu^t\left(\bigoplus_{n  = 0}^\nu F'_n\right) \cap A \subseteq \Lambda_\nu^t(B) = \Lambda^t(B)$.
\end{proof}
Setting $F_n = \C$, $n \in \N$, in Theorem \ref{vector-valued} we obtain the following criterion for scalar-valued Eidelheit sequences:
\begin{corollary}\label{scalar-valued}
Let $E$ be a l.c.s.\ and let $(x'_n)_n \subset E'$. If $(x'_n)_n$ is Eidelheit, then:
\begin{itemize}
\item[(P1)] The set of linear functionals $\{x'_n \, : \, n \in \N\}$ is linearly independent. 
\item[(P2)] For every equicontinuous set $A \subset E'$ the set
$$
\operatorname{span} \{x'_n \, : \, n \in \N\} \cap A
$$
is contained in a finite-dimensional subspace.
\end{itemize}
Conversely, if $E$ is $B$-complete, the properties \textnormal{(P1)} and \textnormal{(P2)} are also sufficient for $(x'_n)_n$ to be an Eidelheit sequence.
\end{corollary}
\begin{proof} Indeed, property $\textbf{(P3)}$ is superfluous for $F_n = \C$, $n \in \N$, since it is always implied by (P1) and the Hahn-Banach theorem.
\end{proof}
When $E$ is a Fr\'echet space, we recover Eidelheit's original theorem \cite{Eidelheit}.
\section{Associated vector-valued Eidelheit sequences}\label{associated}
 We can canonically associate a vector-valued system of equations to each scalar-valued system of equations \eqref{scalar-problem} in the following way:
Let $E$ and $F$ be l.c.s.\ and suppose that $F$ is complete. For a given sequence $(x'_n)_n \subset E'$ we call $(x'_n \widehat{\otimes}_t \operatorname*{Id}_F)_n \subset L(E \widehat{\otimes}_t F, F)$, with $t = \varepsilon, \pi, \beta$, or $i$, the \emph{$F$-associated sequence of $(x'_n)$} (with respect to the $t$-topology). 

In this section we discuss whether the $F$-associated sequence of an Eidelheit sequence is itself Eidelheit. We start by giving sufficient conditions on $E$ and $F$ such that the $F$-associated sequence of any Eidelheit sequence in $E'$ is again Eidelheit.

\begin{proposition}\label{ass-vector-valued} Let $E$ and $F$ be locally convex spaces.
\begin{itemize}
\item[$(i)$ ]Let $t = \varepsilon$ or $\pi$. If $E$ is $B$-complete, $F$ is barreled and complete, and $E \widehat{\otimes}_t F$ is $B$-complete, then for any Eidelheit sequence $(x'_n)_n \subset E'$ its $F$-associated sequence,  with respect to the $t$-topology, is Eidelheit.
\item[$(ii)$] If  $E$ is $B$-complete, $F$ is barreled complete and contains a bounded total set, and $E \widehat{\otimes}_\beta F$ is $B$-complete, then for any Eidelheit sequence $(x'_n)_n \subset E'$ its $F$-associated sequence, with respect to the $\beta$-topology, is Eidelheit.
\end{itemize} 
\end{proposition}
\begin{proof} It suffices to show that $(x'_n \widehat{\otimes}_t \operatorname*{Id}_F)_n$ satisfies condition \textbf{(P1)}, \textbf{(P2)}, and \textbf{(P3)} of Theorem \ref{vector-valued} for $t=\varepsilon,\pi,$ or $\beta$, under the corresponding hypotheses. Notice that 
$$
y' \circ (x' \widehat{\otimes}_t \operatorname{Id}_F) = x' \widehat{\otimes}_t y', \qquad x' \in E',\  y' \in F'.
$$
We simultaneously show that \textbf{(P1)} and \textbf{(P3)} must hold, a case distinction is only necessary in the proof of \textbf{(P2)}.

\textbf{(P1)}: Let $y'_0, \ldots, y'_N \in F'$ be such that
$$
\sum_{n = 0}^N  x'_n \widehat{\otimes}_t y'_n = 0.
$$ 
Since the set $\{x'_n \, : \, n \in \N\}$ is linearly independent, we obtain that $y'_n = 0$ for all $n=0, \ldots, N$. 

\textbf{(P3)}: Let $y_0, \ldots, y_N\in F$ be arbitrary.  Since the set of vectors $\{x'_n \, : \, n \in \N \}$ is linearly independent, there are elements $x_0, \ldots, x_N \in E$ such that 
$$
\langle x'_n, x_k \rangle = \delta_{k,n}, \qquad k,n= 0, \ldots N,
$$ 
where $\delta_{k,n}$ is the Kronecker delta. Hence,
$$
\langle x'_n \widehat{\otimes}_t \operatorname{Id}_F , \sum_{k = 0}^N x_k \otimes y_k \rangle = \sum_{k = 0}^N \langle x'_n, x_k \rangle y_k  = y_n, \qquad n= 0, \ldots , N.
$$

\textbf{(P2)}: We now treat the cases $(i)$ and $(ii)$ separately.

$(i)$ Let $t=\varepsilon$ or $\pi$. Since the canonical bilinear mapping $E \times F \rightarrow E \otimes_t F$ is jointly continuous, it suffices to show that for arbitrary open neighborhoods of the origin $U$ and $V$ in $E$ and $F$, respectively, there is $\nu  \in \N$ such that for all $N > \nu$ and $y'_0, \ldots, y'_N \in F$ it holds that
\begin{equation}
\sum_{n = 0}^N  x'_n \widehat{\otimes}_t y'_n \in (U \otimes V)^\circ
\label{P2-000}
\end{equation}
implies $y'_{\nu + 1} = \ldots =  y'_N = 0$. Let $\nu \in \N$ be such that for all $N > \nu$ and $b_0, \ldots, b_N \in \C$ 
\begin{equation}
\sum_{n = 0}^N b_n x'_n  \in U^\circ
\label{P2-111}
\end{equation}
implies $b_{\nu + 1} = \ldots b_N = 0$. Now suppose that $\eqref{P2-000}$ holds. We have 
$$
\sum_{n = 0}^N  \langle y'_n, y \rangle x'_n   \in U^\circ 
$$
for all $y \in V$. Property \eqref{P2-111} and the fact that $V$ is absorbing therefore imply that $y'_{\nu + 1} = \ldots =  y'_{N} = 0$.

$(ii)$ We now consider $t=\beta$. Let $B$ be a bounded total set in $F$. Since the canonical bilinear mapping $E \times F \rightarrow E \otimes_\beta F$ is hypocontinuous, it suffices to show that for an arbitrary open neighborhood of the origin $U$ in $E$ there is $\nu  \in \N$ such that for all $N > \nu$ and $y'_0, \ldots, y'_N \in F$ it holds that
\begin{equation}
\sum_{n = 0}^N  x'_n \widehat{\otimes}_\beta y'_n \in (U \otimes B)^\circ
\label{P2-0000}
\end{equation}
implies $y'_{\nu + 1} = \ldots =  y'_N = 0$. Let $\nu \in \N$ be such that for all $N > \nu$ and $b_0, \ldots, b_N \in \C$ it holds that
\begin{equation}
\sum_{n = 0}^N b_n x'_n  \in U^\circ
\label{P2-1111}
\end{equation}
implies $b_{\nu + 1} = \ldots b_N = 0$. Now suppose that $\eqref{P2-0000}$ holds. We have 
$$
\sum_{n = 0}^N  \langle y'_n, y \rangle x'_n   \in U^\circ 
$$
for all $y \in B$. Property \eqref{P2-1111} and the fact that $B$ is total therefore imply that $y'_{\nu + 1} = \ldots =  y'_{N} = 0$.
\end{proof}

\begin{corollary}\label{ass-vector-Frechet}
Let $E$ and $F$ be Fr\'echet spaces. For any Eidelheit sequence $(x'_n)_n \subset E'$ its $F$-associated sequence, with respect to both the $\varepsilon$- and $\pi$-topology, is also Eidelheit.
\end{corollary}
\begin{remark}\label{introductoryremark}
For the $\pi$-topology, Corollary \ref{ass-vector-Frechet} can also be shown in the following alternative way: We need to show that the mapping 
$$
\Lambda_F : E \widehat{\otimes}_\pi F \rightarrow F^\N: \Phi \rightarrow (\langle x'_n \widehat{\otimes}_\pi \operatorname{Id}_F, \Phi \rangle)_n
$$ 
is surjective. Notice that $\Lambda_F = \Lambda \widehat{\otimes}_\pi \operatorname*{Id}_F$, where 
$$
\Lambda : E  \rightarrow \C^\N: x \rightarrow (\langle x'_n, x \rangle)_n\: .
$$ 
The surjectivity of $\Lambda_F$ then immediately follows from the ensuing well known fact \cite{Treves}: Given two surjective continuous linear mappings $T_1: E_1 \rightarrow F_1$ and $T_2: E_2 \rightarrow F_2$ between Fr\'echet spaces, the mapping 
$$
T_1 \widehat{\otimes}_\pi T_2: E_1 \widehat{\otimes}_\pi E_2 \rightarrow F_1 \widehat{\otimes}_\pi F_2
$$
is also surjective.
\end{remark}
\begin{remark}
Since $B$-completeness is hard to verify except when implied by formally stronger properties, in practice only Corollary \ref{ass-vector-Frechet} seems to be directly applicable. However, for many classical Eidelheit type problems (e.g.\ the Borel problem), the $(DF)$-spaces $F$ for which the associated $F$-valued problem is solvable have been characterized by dual $(DN)$-$(\Omega)$ type conditions; see \cite{Vogt-77} for the Borel problem, \cite{Braun-87} for holomorphic interpolation, and \cite{Bonet-02} for real analytic interpolation. On the other hand, Proposition \ref{ass-vector-valued} is a useful device for proving non B-completeness, as we show in the rest of the article.
\end{remark}
Next, we are interested in giving sufficient conditions on $E$ and $F$ such that the $F$-associated sequence of any sequence in $E'$ is \emph{not} Eidelheit. The idea is to find conditions on $E$ and $F$ ensuring that their completed tensor product is ``small''. Let us fix some terminology. A l.c.s.\ $F$ is called a {$(LF)$\emph{-space} (resp., $(LB)$\emph{-space}) if $F$ is the locally convex inductive limit of an inductive sequence $(F_n)_{n \in \N}$ of Fr\'echet spaces (resp., Banach spaces). The sequence $(F_n)_n$ is called \emph{strict} if $F_n$ is a topological subspace of  $F_{n+1}$ for each $n \in \N$. In such a case, we also say that $F$ is strict. It is well known that every strict $(LF)$-space $F$ is complete, \emph{regular}, i.e., every bounded set in $F$ is contained and bounded in some step $F_N$, and that, for every $n \in \N$, the induced topology of $F$ on $F_n$ coincides with the original topology on $F_n$. If $F_n \subsetneq F_{n+1}$ for each $n \in \N$ we say that $F$ is \emph{proper}. Finally, we will write $E'_A$ for the linear span of a subset $A \subseteq E'$.
\begin{lemma} \label{solvability}
Let $E$ be a Fr\'echet space. Then, $E$ is non-normable if and only if $E'$ contains an Eidelheit sequence.
\end{lemma}
\begin{proof} Suppose $E$ is non-normable.
Let $(U_n)_{n \in \N}$ be a fundamental system of neighborhoods of the origin in $E$ consisting of closed convex balanced sets such that $U_{n+1} \subseteq U_{n}$ for each $n \in \N$. The non-normability of $E$ and the bipolar theorem imply that there is a sequence of natural numbers $(k_n)_n$ such that $E'_{U^\circ_{k_n}} \subsetneq E'_{U^\circ_{k_{n+1}}}$ for each $n \in \N$. Choose $x'_0 \in E'_{U^\circ_{k_0}}$, $x'_0 \neq 0$, and $x'_n \in E'_{U^\circ_{k_n}} \backslash E'_{U^\circ_{k_{n-1}}}$, $n \in \Z_+$. Clearly, the sequence $(x'_n)_n$ satisfies the conditions of Corollary \ref{scalar-valued}. Conversely, no sequence in a Banach space can be Eidelheit, as follows from Corollary \ref{scalar-valued}.
\end{proof}

\begin{proposition}\label{non-solvable}\mbox{}
\begin{itemize}
\item[$(i)$] Let $E$ be a complete semi-reflexive nuclear l.c.s.\ that admits a continuous norm and let $F = \varinjlim_{n} F_n$ be a proper strict $(LF)$-space. Then, for any sequence in $E'$ its $F$-associated sequence, with respect to both the $\varepsilon$- and $\pi$-topology, is not Eidelheit.
\item[$(ii)$] Let $G$ be a complete semi-reflexive nuclear l.c.s., $H$ a non-normable reflexive Fr\'echet space, and suppose that $G \widehat{\otimes}_\pi H$ is bornological. Set $E = G'_b$ and $F = H'_b$. Then, for any sequence in $E'$ its $F$-associated sequence, with respect to to $i$-topology, is not Eidelheit.
\end{itemize}
\end{proposition}
\begin{proof} $(i)$ Since $E$ is nuclear, there is no distinction between the $\varepsilon$- and $\pi$-topology and we simply write $E \widehat{\otimes} F \cong E \widehat{\otimes}_\varepsilon F \cong E \widehat{\otimes}_\pi F$. Notice that, since $E$ is complete and nuclear, and $F$ complete, we have that $E \widehat{\otimes} F \cong L(E'_b, F)$. Let $(x'_n)_n \subset E'$ be arbitrary. Choose $y_0 \in F_0$ and $y_n \in F_n \backslash F_{n-1}$, $n \in \Z_+$, and suppose that there is $\Phi \in L(E'_b,F)$ such that
\begin{equation}
y_n = x'_n \widehat{\otimes} \operatorname{Id}_F(\Phi) = \Phi(x'_n), \qquad n \in \N.
\label{contradiction}
\end{equation}
The fact that $E$ is semi-reflexive and admits a continuous norm implies that $E'_b$ contains a bounded total set $B$. As strict inductive limits are regular, we have that $\Phi(B) \subset F_N$ for some $N \in \N$. Since $F_N$ is closed in $F$, we obtain that 
$$
\Phi(E') = \Phi\left(\overline{E'_B}\right) \subseteq \overline{\Phi (E'_B)}^F \subseteq \overline{F_N}^F = F_N.
$$
In view of the choice of the sequence $(y_n)_n$ this contradicts \eqref{contradiction}.

$(ii)$ By \cite[Chap.\ II, Corollaire, p.\ 90]{Grothendieck} we have that 
$$
G'_b \widehat{\otimes}_i H'_b \cong (G \widehat{\otimes}_\pi H)' \cong B(G,H),
$$
where $B(G,H)$ denotes the space of jointly continuous bilinear forms on $G \times H$. Let $(g_n)_n \subset E' = G$ be arbitrary. Lemma \ref{solvability} implies that $F = H'$ contains an Eidelheit sequence $(h'_n)_n$. Suppose that there is $\Phi \in B(G,H)$ such that
$$
h'_n = g_n \widehat{\otimes}_i \operatorname{Id}_{H'}(\Phi) = \Phi(g_n,\cdot), \qquad n \in \N.
$$
Since the mapping $\Phi$ is jointly continuous, there are neighborhoods of the origin $U$ and $V$ of $E$ and $F$, respectively, such that 
$$
|\Phi(g,h)| \leq 1, \qquad g \in U, h \in V.
$$
Hence,
$$
h'_n = \Phi(g_n,\cdot) \in \bigcup_{\lambda > 0} \lambda V^\circ, \qquad n \in \N. 
$$
Since we have chosen $(h'_n)_n$ to be Eidelheit, this contradicts property $(P2)$ of Corollary \ref{scalar-valued}.
\end{proof}
\section{Application: Non $B$-completeness of completed tensor products}
In this section we give sufficient conditions on a Fr\'echet space $E$ and a reflexive $(DF)$-space $F$ for $E \widehat{\otimes}_t F$, with $t = \varepsilon, \pi, \beta$, or $i$, not to be  $B$-complete. The results are based on the work of Valdivia \cite{Valdivia-87} and our results from Section \ref{associated}. We mention that the question whether the $\beta$-completed tensor product of two $B$-complete spaces is again $B$-complete was first raised by Summers \cite{Summers}.
\begin{theorem} \label{counterexample} Let $E$ and $F$ be locally convex spaces.\ 
\begin{itemize}
\item[$(i)$] The spaces $E \widehat{\otimes}_\varepsilon F$ and $E \widehat{\otimes}_\pi F$ are not $B$-complete in the following cases:
\begin{itemize}
\item[$(a)$] $E$ is an infinite-dimensional nuclear Fr\'echet space that admits a continuous norm and $F$ is a reflexive proper strict $(LB)$-space. 
\item[$(b)$] $E$ is a Fr\'echet space which does not admit a continuous norm and $F$ is an infinite-dimensional Montel $(DF)$-space such that $F \not \cong \C^{(\N)}$. 
\end{itemize}
\item[$(ii)$] The spaces $E \widehat{\otimes}_\beta F$ and  $E \widehat{\otimes}_i F$ are not $B$-complete in the following cases:
\begin{itemize}
\item[$(a)$] $E$ is an infinite-dimensional nuclear Fr\'echet space, $F$ is a non-normable reflexive $(DF)$-space that contains a bounded total set, and $E'_b \widehat{\otimes}_\pi F'_b$ is bornological. 
\item[$(b)$] $E$ is an infinite-dimensional Fr\'echet-Montel space such that $E \not \cong \C^{\N}$ and $F$ is a reflexive $(DF)$-space that does not contain a bounded total set. 
\end{itemize}
\end{itemize}
On the other hand, in all four cases, the spaces $E$ and $F$ are $B$-complete.
\end{theorem}
In the proof of Theorem \ref{counterexample} we shall employ the following result of Valdivia. Given a l.c.s.\ $E$ we employ the notation
 $$
 E^\N = \prod_{n \in \N} E, \qquad E^{(\N)} = \bigoplus_{n \in \N} E.
 $$
\begin{proposition}\cite[Thm.\ 6]{Valdivia-87} \label{Valdivia-B-complete} Let $E$ be an infinite-dimensional Fr\'echet-Montel space. The following facts are equivalent:
\begin{itemize}
\item[$(i)$] $E^{(\N)}$ is $B$-complete,
\item[$(ii)$] $(E'_b)^{\N}$ is $B$-complete,
\item[$(iii)$] $E \cong \C^\N$.
\end{itemize}
\end{proposition}

\begin{proof}[Proof of Theorem \ref{counterexample}]
$(i)$ Case $(a)$: It follows directly from Lemma \ref{solvability} and Propositions \ref{ass-vector-valued}$(i)$ and \ref{non-solvable}$(i)$.
Case $(b)$: Let $t = \varepsilon$ or $\pi$. By \cite[Chap.\ II, Lemme 10.2, p.\ 93]{Grothendieck} the space $E$ contains a complemented subspace that is isomorphic to $\C^\N$. Hence the space $E \widehat{\otimes}_t F$ contains a complemented subspace that is isomorphic to $\C^\N \widehat{\otimes}_t F$. As closed subspaces of $B$-complete spaces are $B$-complete \cite[Chap.\ 3, Prop.\ 17.4]{Horvath}, it suffices to show that $\C^\N \widehat{\otimes}_t F$ is not $B$-complete. Set $G = F'_b$, an infinite-dimensional Fr\'echet-Montel space such that $G \not \cong \C^{\N}$. Since,
$$
\C^\N \widehat{\otimes}_t F \cong F^{\N} \cong (G'_b)^\N
$$
as l.c.s., this follows from Proposition \ref{Valdivia-B-complete}.

$(ii)$ Since, in both cases, the spaces $E$ and $F$ are barreled, we have that $E \widehat{\otimes}_\beta F = E \widehat{\otimes}_i F$. Therefore we shall simply denote this space by $E \bar{\otimes} F$.
 Case $(a)$: It follows directly from Lemma \ref{solvability} and Propositions \ref{ass-vector-valued}$(ii)$ and \ref{non-solvable}$(ii)$.
Case $(b)$: By \cite[Chap.\ II, Lemme 10.1, p.\ 93]{Grothendieck} the space $F$ contains a complemented subspace that is isomorphic to $\C^{(\N)}$. Hence the space $E \bar{\otimes} F$ contains a complemented subspace that is isomorphic to $E \bar{\otimes} \C^{(\N)}$. As closed subspaces of $B$-complete spaces are $B$-complete, it suffices to show that $E \bar{\otimes} \C^{(\N)}$ is not $B$-complete. But
$$
E \bar{\otimes} \C^{(\N)} \cong E^{(\N)} 
$$
as l.c.s., so that the result follows from Proposition \ref{Valdivia-B-complete}.
\end{proof}
As an application of Theorem \ref{counterexample}, we now show that the space
$$
\mathcal{O}_C(\R^d) = \{ f \in C^{\infty}(\R^d) \, : \, \exists N \in \N \, \forall \alpha \in \N^d : \sup_{x \in \R^d}(1+|x|)^{-N} |f^{(\alpha)}(x)| < \infty \} 
$$ 
of very slowly increasing functions endowed with its canonical $(LF)$-topology is not B-complete. We denote by $s$ the space of rapidly decreasing sequences. 
\begin{proposition}
The space $\mathcal{O}_C(\R^d)$ is not $B$-complete.
\end{proposition}
\begin{proof}
Valdivia \cite{Valdivia-81} has shown that 
$$
\mathcal{O}_M(\R^d) \cong s \widehat{\otimes}_\pi s',
$$
as l.c.s., where $\mathcal{O}_M(\R^d)$ denotes the space of slowly increasing functions. Since $\mathcal{O}'_M(\R^d) \cong \mathcal{O}_C(\R^d)$, as l.c.s., via the Fourier transform and the space $s \widehat{\otimes}_\pi s'$ is bornological \cite[Chap.\ II, Corollaire 2, p.\ 128]{Grothendieck} (see \cite{L-W} for a modern proof using homological algebra techniques), \cite[Chap.\ II, Corollaire, p.\ 90]{Grothendieck} implies that 
$$
\mathcal{O}_C(\R^d) \cong s' \widehat{\otimes}_i s.
$$
Hence $\mathcal{O}_C(\R^d)$ is not $B$-complete by Theorem \ref{counterexample}$(ii)(a)$.
\end{proof}
We end this article by giving an overview of the lack of $B$-completeness of various spaces occurring in the theory of (ultra)distributions. Concerning ultradistributions, we follow the Braun-Meise-Taylor approach \cite{B-M-T}. Since we shall employ sequential representations, we need to introduce power series spaces \cite[Chap.\ 29]{M-V}. Let $\alpha = (\alpha_n)_{n \in \N}$ be a sequence of positive real numbers tending monotonically increasing to infinity. For $h \in \R$ we set
$$
\Lambda_h^\alpha = \{ (c_n)_n \in \C^\N \, : \, |(c_n)_n|_h := \sup_{n \in \N} |c_n|e^{h\alpha_n} < \infty\}.
$$
Define the Fr\'echet spaces
$$
\Lambda_0(\alpha) = \varprojlim_{h \to 0^+}\Lambda^\alpha_{-h},  \qquad \Lambda_\infty(\alpha) = \varprojlim_{h \to \infty}\Lambda^\alpha_{h}.
$$
The space $\Lambda_0(\alpha)$ (resp.,\ $\Lambda_\infty(\alpha)$) is nuclear \cite[Prop.\ 29.6]{M-V} if 
$$
\lim_{n \to \infty} \frac{\log n}{\alpha_n} = 0 \qquad \left( \mbox{resp., }  \sup_{n \in \N} \frac{\log n}{\alpha_n} < \infty \right).
$$
Their strong duals are given by
$$
\Lambda'_0(\alpha) = \varinjlim_{h \to 0^+}\Lambda^\alpha_{h},  \qquad \Lambda'_\infty(\alpha) = \varinjlim_{h \to \infty}\Lambda^\alpha_{-h}.
$$
\begin{proposition} \label{distributionsandB} Let $\Omega \subseteq \R^d$ be an open connected set and let $\omega$ be a non-quasianalytic weight function  in the sense of \cite{B-M-T}. The spaces $\mathcal{D}(\Omega)$, $\mathcal{D}'(\Omega)$, $\mathcal{D}_{(\omega)}(\Omega)$, $\mathcal{D}'_{(\omega)}(\Omega)$, $\mathcal{E}_{\{\omega\}}(\Omega)$, and $\mathcal{E}'_{\{\omega\}}(\Omega)$ are not $B$-complete. 
\end{proposition}
\begin{proof} Set $\alpha_n = \omega(n^{1/d})$, $n \in \N$. The assertion is a direct consequence of Theorem \ref{counterexample} and the ensuing sequential representations due to Valdivia and Vogt \cite{Valdivia-78, Vogt-83a}:
$$
\mathcal{D}(\Omega) \cong s \widehat{\otimes}_i \C^{(\N)}, \qquad \mathcal{D}'(\Omega) \cong s' \widehat{\otimes}_\pi \C^{\N},
$$
$$
\mathcal{D}_{(\omega)}(\Omega) \cong \Lambda_\infty(\alpha) \widehat{\otimes}_i \C^{(\N)}, \qquad \mathcal{D}'_{(\omega)}(\Omega) \cong \Lambda'_\infty(\alpha) \widehat{\otimes}_\pi \C^{\N}, 
$$
$$
\mathcal{E}_{\{\omega\}}(\Omega) \cong \Lambda'_0(\alpha) \widehat{\otimes}_\pi \C^{\N}, \qquad \mathcal{E}'_{\{\omega\}}(\Omega) \cong \Lambda_0(\alpha) \widehat{\otimes}_i \C^{(\N)}. 
$$
\end{proof}

\begin{remark}
The question whether the spaces $\mathcal{D}(\Omega)$ and $\mathcal{D}'(\Omega)$ are $B$-complete was first posed by Raikov and attracted the attention of many authors \cite{H-H, Smolyanov-69, Smolyanov-71, Valdivia-74, Valdivia-77}. The non $B$-completeness of $\mathcal{D}_{(\omega)}(\Omega)$ and $\mathcal{D}'_{(\omega)}(\Omega)$ is due to Valdivia \cite{Valdivia-87}. Furthermore, the lack of $B$-completeness of $\mathcal{E}_{\{\omega\}}(\Omega)$ and $\mathcal{E}'_{\{\omega\}}(\Omega)$ follows directly from their sequential representations and Proposition \ref{Valdivia-B-complete}, and is therefore implicitly due to Valdivia and Vogt. However, the fact that the space $\mathcal{O}_C(\R^d)$ is not $B$-complete seems to be new and does not follow from Valdivia's work.
\end{remark}

\end{document}